\documentclass[amstex,10pt,reqno]{amsart}
\usepackage{amsmath,amsfonts,amssymb,amsthm,enumerate,hyperref,multicol,cite}
\usepackage{graphicx}
\usepackage[english]{babel}
\newtheorem{theorem}{Theorem}
\newtheorem{lemma}{Lemma}

\newtheorem{remark}{Remark}
\numberwithin{equation}{section}
\newcommand\norm[1]{\left\lVert#1\right\rVert}
\usepackage[centering, bottom=1.0in,top=1.0in]{geometry}

\begin{document}
\dedicatory{Dedicated to Prof. R. P. Agarwal on his $74^{th}$ birthday}

\leftline{ \scriptsize \it  }
\title[]
{On some summability methods for a $q$-analogue of an integral type operator based on multivariate $q$-Lagrange polynomials}
\maketitle
\begin{center}
{\bf Purshottam Narain Agrawal$^1$, Rahul Shukla$^2$,  and Behar Baxhaku $^3$}
\vskip0.2in
$^{1,2}$Department of Mathematics\\
Indian Institute of Technology Roorkee\\
Roorkee-247667, India\\
$^1$  email: pnappfma@gmail.com\\
$^2$ email: rshukla@ma.iitr.ac.in\\
$^3$ Department of Mathematics\\
University of Prishtina\\
Prishtina, Kosovo\\
$^3$email: behar.baxhaku@uni-pr.edu
\end{center}
\begin{abstract}
The present paper considers a $q$-analogue of an operator defined by Erku\c{s}-Duman et al. (Calcolo, 45(1) (2008), 53-67) involving $q$-Lagrange polynomials in several variables. The Korovkin type theorems in the settings of deferred weighted A-statistical convergence and the power series method are investigated.

\end{abstract}

\section{Introduction}
The past two decades have witnessed to the constructions of linear positive operators by means of multivariate-Lagrange polynomials (and their q-analgoue) and to their approximation behaviour. Let $\big(\mathfrak{C}(I), \norm{.}\big)$ be the Banach space of all continuous functions on $I=[0,1]$ with the sup-norm $\norm{.}$. After the introduction of celebrated multivariate Lagrange polynomials (widely known as Chan-Chyan-Srivastava polynomials) \cite{Chan}, for $f\in \mathfrak{C}(I)$, Erkus et al. \cite{Erkus} defined the following sequence of linear positive operators as;
\begin{eqnarray}\label{eq2}
\mathfrak{L}_n^{\beta^{(1)},\dotsb,\beta^{(r)}}(f(s);x)=\left\{\prod\limits_{k=1}^{r}\left(1-x\beta_n^{(k)}\right)^n\right\}\sum\limits_{p=0}^{\infty}\left\{\sum\limits_{l_1+l_2+\dotsb+l_r=p}f\left(\frac{l_r}{n+l_r-1}\right)\prod\limits_{s=1}^{r}\left(\beta_n^{(s)}\right)^{l_s}\frac{(n)_{l_s}}{l_s!}\right\}x^p,
\end{eqnarray}
where $x\in I$, $\beta^{(j)}=\big<\beta_n^{(j)}\big>$ are sequences in $(0,1)$ for each $j=1,2,\dotsb~r$, and $(\rho)_s$ denotes the standard Pochhammer symbol. The authors studied the approximation properties of the above operator in the $A$-statistical settings. For every $q \in \mathbb{R}$ such that $|q|<1$ and $n \in \mathbb{N}^{0} =\{0,1,2,...\}$, the $q$-Pochhammer symbol $(\rho;q)_{n}$ is given by
$$(\rho;q)_n=
\begin{array}{cc}
  \bigg\{\begin{array}{cc}
  1 ,& \text{if} \quad n=0 ,\\
  (1-\rho)(1-\rho q)...(1-\rho q^{n-1}),& \text{if} \quad n\in \mathbb{N}, \\
    \end{array}
\end{array}
$$
and the q-analogue of a natural number ($q$-integers) is defined by
$$[n]_q=\frac{1-q^n}{1-q}=1+q+q^2+....+q^{n-1}.$$ 
Altin et al.\cite{Altin} proposed a $q$-multivariable Lagrange polynomials $h_{n,q}^{(\eta_1,\dotsb~\eta_r)}(z_1,z_2,\dotsb~z_r)$ as follows,
\begin{eqnarray}\label{eq3}
h_{n,q}^{(\eta_1,\dotsb~\eta_r)}(z_1,z_2,\dotsb~z_r)=\sum\limits_{l_1+l_1+\dotsb+l_r=n}\left\{\prod\limits_{k=1}^{r}\left(q^{\eta_{k}},q\right)_{l_k}\frac{(z_k)^{l_k}}{\left(q,q\right)_{l_k}}\right\},
\end{eqnarray}
whereas the above multivariate polynomials has the generating function of the following form
\begin{eqnarray}\label{eq3*}
 \prod\limits_{k=1}^{r}\frac{1}{(tz_k;q)_{\eta_{k}}}=\sum\limits_{n=0}^{\infty}h_{n,q}^{(\eta_1,\dotsb~\eta_r)}(z_1,z_2,\dotsb~z_r)t^{n},
\end{eqnarray}
where $|t|<\min\{|z_1|^{-1},\dotsb~|z_r|^{-1}\}$.

Erku\c{s}-Duman et al. \cite{ED2} proposed an integral type generalizations of the operator (\ref{eq2}) in the following manner:
\begin{eqnarray}\label{new}
\mathfrak{E}_n^{\beta^{(1)},\dotsb,\beta^{(r)}}(f(s);x)&=&\left\{\prod\limits_{k=1}^{r}\left(1-x\beta_n^{(k)}\right)^n\right\}\sum\limits_{p=0}^{\infty}\Bigg{\lbrace}\sum\limits_{l_1+l_2+\dotsb+l_r=p}(n+l_r-1)\bigg(\prod\limits_{s=1}^{r}\left(\beta_n^{(s)}\right)^{l_s}\frac{(n)_{l_s}}{l_s!}\bigg)\nonumber\\
&&\int_{\frac{l_{r}}{n+l_{r}-1}}^{\frac{l_{r}+1}{n+l_{r}-1}} f(s)ds\Bigg{\rbrace}x^p,
\end{eqnarray}
and studied its statistical approximation properties by means of modulus of continuity and Peetre's K-functional. 
Using the generating function given by (\ref{eq3*}), Erku\c{s}-Duman \cite{ED} studied the following $q$-analogue of the operator $\mathfrak{L}_{n}^{\beta^{(1)},\dotsb~\beta^{(r)}}$     
\begin{eqnarray}\label{eq4*}
\mathfrak{S}_{n,q}^{\beta^{(1)},\dotsb,\beta^{(r)}}(f(s);x) = \left\{\prod\limits_{k=1}^{r}(x\beta_n^{(k)};q)_n\right\}
\sum\limits_{p=0}^{\infty}\bigg\{\sum\limits_{l_1+l_2+\dotsb+l_r=p}(q^{n};q)_{l_1}(q^{n};q)_{l_2}...(q^{n};q)_{l_r}\nonumber\\\hspace{2cm}\frac{(\beta_n^{(1)})^{l_1}(\beta_n^{(2)})^{l_2}...(\beta_n^{(r)})^{l_r}}{(q;q)_{l_1}(q;q)_{l_2}...(q;q)_{l_r}}f\left(\frac{[l_{r}]_{q}}{[n+l_{r}-1]_{q}}\right)\bigg\}x^p.
\end{eqnarray}

It has been observe that the above operator (\ref{eq4*}) was also indepedently tackled by Mursaleen et al. \cite{Mursalen} but unfortunately the proposed definition was incorrect while Behar et al. \cite{Bax} extended the study of Erku\c{s}-Duman to the bi-variate and GBS (Generalized Boolean Sum) cases.  
The prime objective of this paper is to define a $q$-analogue of the operators (\ref{new}) by means of Riemann type $q$-integral, and to study the convergence of such operators via summability methods. In Section 2., we construct the operator and establish some important lemmas to prove the main results. In Section 3., the Korovkin type theorems in the deferred weighted $A$-statistical approximation are studied for these operators. In the last Section, we establish the basic convergence theorem and an estimate of error in the approximation by using power series summability method.
\section{Construction of the operators and Important Lemmas}
Marinkovi\'{c} et al. \cite{Marin} introduced the following Riemann type $q$-integral
\begin{eqnarray}\label{Riman1}
\int_\alpha^\beta f(s) {\rm d_q^R}s = (1-q)(\alpha-\beta)\sum_{j=0}^{\infty} f\left(\alpha+(\beta-\alpha)q^j\right)q^j,
\end{eqnarray}
where $\alpha, \beta, q \in \mathbb{R}$ such that $0<\alpha<\beta$ and $q \in (0,1)$.
This definition of $q$-integral is appropriate to derive some $q$-analogues of well-known integral inequalities \cite{MarinkovicRajkoviStankovi2008}.
Using the $q$-Riemann type integral, for $f\in \mathfrak{C}(I)$, we propose a $q$-analogue of the operators (\ref{new}) as follows:
\begin{eqnarray}\label{eq5}
\mathfrak{K}_{n,q}^{\beta^{(1)},\dotsb,\beta^{(r)}}(f(s);x) = \left\{\prod\limits_{k=1}^{r}(x\beta_n^{(k)};q)_n\right\}
\sum\limits_{p=0}^{\infty}\bigg\{\sum\limits_{l_1+l_2+\dotsb+l_r=p}(q^{n};q)_{l_1}(q^{n};q)_{l_2}...(q^{n};q)_{l_r}\nonumber\\\hspace{2cm}\times[n+l_r-1]_q q^{-l_r}\frac{(\beta_n^{(1)})^{l_1}(\beta_n^{(2)})^{l_2}...(\beta_n^{(r)})^{l_r}}{(q;q)_{l_1}(q;q)_{l_2}...(q;q)_{l_r}}\int_{\frac{[l_{r}]_{q}}{[n+l_{r}-1]_{q}}}^{\frac{[l_{r}+1]_{q}}{[n+l_{r}-1]_{q}}} f(s) {\rm d_q^R}s  \bigg\}x^p.
\end{eqnarray}

\newpage
In order to discuss our main results, we first give the following lemmas.
\begin{lemma}\label{lem1} The operators $\mathfrak{K}_{n,q}^{\beta^{(1)},\dotsb,\beta^{(r)}}(.;x)$ verify the assertions:
\begin{enumerate}[(i)]
\item $\mathfrak{K}_{n,q}^{\beta^{(1)},\dotsb,\beta^{(r)}}(1;x)=1$;\label{part1}
\item $\mathfrak{K}_{n,q}^{\beta^{(1)},\dotsb,\beta^{(r)}}(s;x) ~ \leq ~  x\beta_{n}^{(r)}+\frac{1}{[2]_q[n]_{q}}$.  Moreover, $$\big|\mathfrak{K}_{n,q}^{\beta^{(1)},\dotsb,\beta^{(r)}}(s;x)-x\big| ~\leq ~ x(1-\beta_{n}^{(r)})+ \frac{1}{[2]_q[n]_{q}}.$$
\item $\mathfrak{S}_{n,q}^{\beta^{(1)},\dotsb,\beta^{(r)}}(s^{2};x) ~ \leq ~ \frac{1}{[3]_q[n]_q^2}+\frac{x\beta_{n}^{(r)}}{[n]_q}\left(1+\frac{2}{[2]_q}\right) + q (x\beta_{n}^{(r)})^{2}$. Also, $$\big|\mathfrak{K}_{n,q}^{\beta^{(1)},\dotsb,\beta^{(r)}}(s^{2};x)-x^{2}\big| ~\leq ~ \frac{1}{[3]_q[n]_q^2}+\frac{x\beta_{n}^{(r)}}{[n]_q}\left(1+\frac{2}{[2]_q}\right) + 2x^{2}(1-\beta_{n}^{(r)}).$$
\end{enumerate}
\end{lemma}
\begin{proof}
\begin{enumerate}[(i)]
\item Using the definition of $q$-Riemann integral given by (\ref{Riman1}) and combining (\ref{eq3})-(\ref{eq3*}), we have the result.
\item From (\ref{eq4*}), we have
\begin{eqnarray*}
\mathfrak{K}_{n,q}^{\beta^{(1)},\dotsb,\beta^{(r)}}(s;x)&=& \left\{\prod\limits_{k=1}^{r}(x\beta_n^{(k)};q)_n\right\}\sum\limits_{p=0}^{\infty}\bigg\{\sum\limits_{l_1+\dotsb+l_r=p}(q^{n};q)_{l_1}\dotsb(q^{n};q)_{l_r}[n+l_r-1]_q q^{-l_r}\\&\times&\frac{(\beta_n^{(1)})^{l_1}\dotsb(\beta_n^{(r)})^{l_r}}{(q;q)_{l_1}\dotsb(q;q)_{l_r}}\int_{\frac{[l_{r}]_{q}}{[n+l_{r}-1]_{q}}}^{\frac{[l_{r}+1]_{q}}{[n+l_{r}-1]_{q}}} s~ {\rm d_q^R}s\bigg\}x^p\\&=&\left\{\prod\limits_{k=1}^{r}(x\beta_n^{(k)};q)_n\right\}\sum\limits_{p=0}^{\infty}\bigg\{\sum\limits_{l_1+\dotsb+l_r=p}(q^{n};q)_{l_1}\dotsb(q^{n};q)_{l_r}[n+l_r-1]_q q^{-l_r}\\&\times&\frac{(\beta_n^{(1)})^{l_1}\dotsb(\beta_n^{(r)})^{l_r}}{(q;q)_{l_1}\dotsb(q;q)_{l_r}}\frac{1}{[n+l_r-1]_q^2}\big\{q^{l_r}[l_r]_q+\frac{q^{2l_r}}{[2]_q}\big\}\bigg\}x^p\\
&=&\left\{\prod\limits_{k=1}^{r}(x\beta_n^{(k)};q)_n\right\}\sum\limits_{p=0}^{\infty}\bigg\{{\sum\limits_{l_1+\dotsb+l_r=p}}(q^{n};q)_{l_1}\dotsb(q^{n};q)_{l_r}\frac{[l_{r}]_{q}}{[n+l_{r}-1]_{q}}\frac{(\beta_n^{(1)})^{l_1}\dotsb(\beta_n^{(r)})^{l_r}}{(q;q)_{l_1}\dotsb(q;q)_{l_r}}\bigg\}x^p\\
&+&\left\{\prod\limits_{k=1}^{r}(x\beta_n^{(k)};q)_n\right\}\sum\limits_{p=0}^{\infty}\bigg\{{\sum\limits_{l_1+\dotsb+l_r=p}}(q^{n};q)_{l_1}\dotsb(q^{n};q)_{l_r}\frac{q^{l_r}}{[2]_q[n+l_{r}-1]_{q}}\frac{(\beta_n^{(1)})^{l_1}\dotsb(\beta_n^{(r)})^{l_r}}{(q;q)_{l_1}\dotsb(q;q)_{l_r}}\bigg\}x^p.
\end{eqnarray*}
Since, $\frac{[l_r]_{q}}{(q;q)_{l_r}}=\frac{1}{(1-q)(q;q)_{l_r-1}},$ $\frac{(q^{n};q)_{l_r}}{[n+l_r-1]_{q}}=(1-q)(q^{n};q)_{l_r-1}$ and $\frac{q^{l_r}}{[n+l_r-1]_{q}} \leq \frac{1}{[n]_q}$, we have
\begin{eqnarray*}
\mathfrak{K}_{n,q}^{\beta^{(1)},\dotsb,\beta^{(r)}}(s;x) &\leq& x\beta_{n}^{(r)} \left\{\prod\limits_{k=1}^{r}(x\beta_n^{(k)};q)_n\right\}\sum\limits_{p=1}^{\infty}\bigg\{\underset{l_r\geq 1}{\sum\limits_{l_1+\dotsb+l_r-1=p-1}}(q^{n};q)_{l_1}\dotsb(q^{n};q)_{l_r-1} \\
&&\frac{(\beta_n^{(1)})^{l_1}\dotsb(\beta_n^{(r)})^{l_r-1}}{(q;q)_{l_1}\dotsb(q;q)_{l_r-1}}\bigg\}x^{p-1}\\&+&\frac{1}{[2]_q[n]_{q}}\left\{\prod\limits_{k=1}^{r}(x\beta_n^{(k)};q)_n\right\}\sum\limits_{p=0}^{\infty}\bigg\{{\sum\limits_{l_1+\dotsb+l_r=p}}(q^{n};q)_{l_1}\dotsb(q^{n};q)_{l_r}\frac{(\beta_n^{(1)})^{l_1}\dotsb(\beta_n^{(r)})^{l_r}}{(q;q)_{l_1}\dotsb(q;q)_{l_r}}\bigg\}x^p\\
&\leq& x\beta_{n}^{(r)} \left\{\prod\limits_{k=1}^{r}(x\beta_n^{(k)};q)_n\right\}\sum\limits_{p=1}^{\infty}h^{(n,\dotsb,n)}_{p-1,q}(\beta_n^{(1)},\beta_n^{(2)},\dotsb, \beta_n^{(r)})x^{p-1}\\&+& \frac{1}{[2]_q[n]_{q}}\left\{\prod\limits_{k=1}^{r}(x\beta_n^{(k)};q)_n\right\}\sum\limits_{p=0}^{\infty}h^{(n,\dotsb,n)}_{p,q}(\beta_n^{(1)},\beta_n^{(2)},\dotsb, \beta_n^{(r)})x^{p}, \quad \text{from}~ (\ref{eq3})\\
&\leq & x\beta_{n}^{(r)}+\frac{1}{[2]_q[n]_{q}}, \quad \text{in view of}~ (\ref{part1}).
\end{eqnarray*}
Using above inequality we can get
\begin{eqnarray}\label{Riman2}
\mathfrak{K}_{n,q}^{\beta^{(1)},\dotsb,\beta^{(r)}}(s;x)-x\leq x\left(\beta_{n}^{(r)}-1\right)+\frac{1}{[2]_q[n]_{q}}.
\end{eqnarray}
On the other hand, we have
\begin{eqnarray}\label{Riman3}
\mathfrak{K}_{n,q}^{\beta^{(1)},\dotsb,\beta^{(r)}}(s;x)&\geq&\left\{\prod\limits_{k=1}^{r}(x\beta_n^{(k)};q)_n\right\}\sum\limits_{p=1}^{\infty}\bigg\{\underset{l_r\geq 1}{\sum\limits_{l_1+\dotsb+l_r=p}}(q^{n};q)_{l_1}\dotsb(q^{n};q)_{l_r}\frac{[l_{r}]_{q}}{[n+l_{r}-1]_{q}}\frac{(\beta_n^{(1)})^{l_1}\dotsb(\beta_n^{(r)})^{l_r}}{(q;q)_{l_1}\dotsb(q;q)_{l_r}}\bigg\}x^p\nonumber\\&\geq& x\beta_{n}^{(r)}.
\end{eqnarray}
Thus by equation (\ref{Riman2}) and (\ref{Riman3}), we get
$$\big|\mathfrak{K}_{n,q}^{\beta^{(1)},\dotsb,\beta^{(r)}}(s;x)-x\big|~\leq ~ x(1-\beta_{n}^{(r)})+ \frac{1}{[2]_q[n]_{q}}.$$
\item Since from $(\ref{eq4*})$, we have
\begin{eqnarray*}
\mathfrak{K}_{n,q}^{\beta^{(1)},\dotsb,\beta^{(r)}}(s^{2};x) &=& \left\{\prod\limits_{k=1}^{r}(x\beta_n^{(k)};q)_n\right\}\sum\limits_{p=0}^{\infty}\bigg\{\sum\limits_{l_1+\dotsb+l_r=p}(q^{n};q)_{l_1}\dotsb(q^{n};q)_{l_r}\nonumber\\
&&\times[n+l_r-1]_q q^{-l_r}\frac{(\beta_n^{(1)})^{l_1}\dotsb(\beta_n^{(r)})^{l_r}}{(q;q)_{l_1}\dotsb(q;q)_{l_r}}\int_{\frac{[l_{r}]_{q}}{[n+l_{r}-1]_{q}}}^{\frac{[l_{r}+1]_{q}}{[n+l_{r}-1]_{q}}} s^2 {\rm d_q^R}s\bigg\}x^p\nonumber
\end{eqnarray*}
\begin{eqnarray*}
\mathfrak{K}_{n,q}^{\beta^{(1)},\dotsb,\beta^{(r)}}(s^{2};x)&=&\left\{\prod\limits_{k=1}^{r}(x\beta_n^{(k)};q)_n\right\}\sum\limits_{p=0}^{\infty}\bigg\{{\sum\limits_{l_1+\dotsb+l_r=p}}(q^{n};q)_{l_1}\dotsb(q^{n};q)_{l_r}[n+l_r-1]_q q^{-l_r}\nonumber\\
&&\times \frac{1}{[n+l_{r}-1]_{q}^3}\big\{q^{l_r}[l_r]_q^2+\frac{2q^{2l_r}[l_r]_q}{[2]_q}+\frac{q^{3l_r}}{[3]_q}\big\}\frac{(\beta_n^{(1)})^{l_1}\dotsb(\beta_n^{(r)})^{l_r}}{(q;q)_{l_1}\dotsb(q;q)_{l_r}}\bigg\}x^p\nonumber
\end{eqnarray*}
\begin{eqnarray*}
&=& x\beta_{n}^{(r)} \left\{\prod\limits_{k=1}^{r}(x\beta_n^{(k)};q)_n\right\}\sum\limits_{p=1}^{\infty}\bigg\{\underset{l_r\geq 1}{\sum\limits_{l_1+\dotsb+l_r-1=p-1}}(q^{n};q)_{l_1}\dotsb(q^{n};q)_{l_r-1}\nonumber\\
&&\bigg(\frac{1+q[l_{r}-1]_{q}}{[n+l_{r}-1]_{q}}\bigg)\frac{(\beta_n^{(1)})^{l_1}\dotsb(\beta_n^{(r)})^{l_r-1}}{(q;q)_{l_1}\dotsb(q;q)_{l_r-1}}\bigg\}x^{p-1}\nonumber\\&+& \frac{2}{[2]_q} \left\{\prod\limits_{k=1}^{r}(x\beta_n^{(k)};q)_n\right\}\sum\limits_{p=0}^{\infty}\bigg\{{\sum\limits_{l_1+\dotsb+l_r=p}}(q^{n};q)_{l_1}\dotsb(q^{n};q)_{l_r}\nonumber\\
&&\frac{q^{l_r}[l_{r}]_{q}}{[n+l_{r}-1]_{q}^2}\frac{(\beta_n^{(1)})^{l_1}\dotsb(\beta_n^{(r)})^{l_r}}{(q;q)_{l_1}\dotsb(q;q)_{l_r}}\bigg\}x^{p}\nonumber\\&+& \frac{1}{[3]_q} \left\{\prod\limits_{k=1}^{r}(x\beta_n^{(k)};q)_n\right\}\sum\limits_{p=0}^{\infty}\bigg\{{\sum\limits_{l_1+\dotsb+l_r=p}}(q^{n};q)_{l_1}\dotsb(q^{n};q)_{l_r}\nonumber\\
&&\frac{q^{2l_r}}{[n+l_{r}-1]_{q}^2}\frac{(\beta_n^{(1)})^{l_1}\dotsb(\beta_n^{(r)})^{l_r}}{(q;q)_{l_1}\dotsb(q;q)_{l_r}}\bigg\}x^{p}\nonumber\\
&=& \sum_1+\sum_2+\sum_3+\sum_4 \quad \text{say.}
\end{eqnarray*}
Now,
\begin{eqnarray}\label{eq4**}
\sum_{1}&=& x\beta_{n}^{(r)} \left\{\prod\limits_{k=1}^{r}(x\beta_n^{(k)};q)_n\right\}\sum\limits_{p=1}^{\infty}\bigg\{\underset{l_r\geq 1}{\sum\limits_{l_1+\dotsb+l_r-1=p-1}}(q^{n};q)_{l_1}\dotsb(q^{n};q)_{l_r-1}\frac{1}{[n+l_r-1]_{q}}\nonumber\\
&&\frac{(\beta_n^{(1)})^{l_1}\dotsb(\beta_n^{(r)})^{l_r-1}}{(q;q)_{l_1}\dotsb(q;q)_{l_r-1}}\bigg\}x^{p-1}\nonumber\\
&\leq & \frac{x\beta_{n}^{(r)}}{[n]_q},\quad \text{using}~ \frac{1}{[n+l_r-1]_{q}} \leq \frac{1}{[n]_q}~ \text{and}~(\ref{part1}).
\end{eqnarray}

Also,
\begin{eqnarray*}
\sum_{2}&=& q x\beta_{n}^{(r)} \left\{\prod\limits_{k=1}^{r}(x\beta_n^{(k)};q)_n\right\}\sum\limits_{p=1}^{\infty}\bigg\{\underset{l_r\geq 1}{\sum\limits_{l_1+\dotsb+l_r-1=p-1}}(q^{n};q)_{l_1}\dotsb(q^{n};q)_{l_r-1}\nonumber\\
&&\frac{[l_r-1]_{q}}{[n+l_r-1]_{q}}\frac{(\beta_n^{(1)})^{l_1}\dotsb(\beta_n^{(r)})^{l_r-1}}{(q;q)_{l_1}\dotsb(q;q)_{l_r-1}}\bigg\}x^{p-1}\\
&=& q (x\beta_{n}^{(r)})^2 \left\{\prod\limits_{k=1}^{r}(x\beta_n^{(k)};q)_n\right\}\sum\limits_{p=2}^{\infty}\bigg\{\underset{l_r\geq 2}{\sum\limits_{l_1+\dotsb+l_r-2=p-2}}(q^{n};q)_{l_1}\dotsb(q^{n};q)_{l_r-2}\\
&&\frac{1-q^{l_r-1}}{(1-q)[n+l_r-1]_{q}}\frac{(1-q^{n+l_r-2})}{1-q^{l_r-1}}\frac{(\beta_n^{(1)})^{l_1}\dotsb(\beta_n^{(r)})^{l_r-2}}{(q;q)_{l_1}\dotsb(q;q)_{l_r-2}}\bigg\}x^{p-2}
\end{eqnarray*}
\begin{eqnarray*}
&=& q (x\beta_{n}^{(r)})^2 \left\{\prod\limits_{k=1}^{r}(x\beta_n^{(k)};q)_n\right\}\sum\limits_{p=2}^{\infty}\bigg\{\underset{l_r\geq 2}{\sum\limits_{l_1+\dotsb+l_r-2=p-2}}(q^{n};q)_{l_1}\dotsb(q^{n};q)_{l_r-2}\\
&&\frac{[n+l_r-2]_{q}}{[n+l_r-1]_{q}}\frac{(\beta_n^{(1)})^{l_1}\dotsb(\beta_n^{(r)})^{l_r-2}}{(q;q)_{l_1}\dotsb(q;q)_{l_r-2}}\bigg\}x^{p-2}.
\end{eqnarray*}
Since $\frac{[n+l_r-2]_{q}}{[n+l_r-1]_{q}} < 1$, using (\ref{part1}), we get
\begin{eqnarray}\label{eq4***}
\sum_{2} &\leq & q (x\beta_{n}^{(r)})^{2}.
\end{eqnarray}
Similarly, using $\frac{1}{[n+l_r-1]_{q}} \leq \frac{1}{[n]_q}$ and $\frac{1}{[n+l_r-1]_{q}^2} \leq \frac{1}{[n]_q^2}$,  we obtain
\begin{eqnarray}\label{Riman4}
\sum_3 &=& x\beta_{n}^{(r)}\frac{2}{[2]_q[n]_q},
\end{eqnarray}
and
\begin{eqnarray}\label{Riman5}
\sum_4 &=& \frac{1}{[3]_q[n]_q^2}, \quad \text{respectively}.
\end{eqnarray}
Finally, combining the equations $(\ref{eq4**})-(\ref{Riman5})$, we have
\begin{eqnarray}\label{eq6}
\mathfrak{K}_{n,q}^{\beta^{(1)},\dotsb,\beta^{(r)}}(s^{2};x) & \leq &\frac{1}{[3]_q[n]_q^2}+\frac{x\beta_{n}^{(r)}}{[n]_q}\left(1+\frac{2}{[2]_q}\right) + q (x\beta_{n}^{(r)})^{2}.
\end{eqnarray}
Now, from $(\ref{eq6})$
\begin{eqnarray*}
\mathfrak{K}_{n,q}^{\beta^{(1)},\dotsb,\beta^{(r)}}(s^{2};x)-x^{2} & \leq & x^{2} (q(\beta_{n}^{(r)})^{2}-1)+ \frac{1}{[3]_q[n]_q^2}+\frac{x\beta_{n}^{(r)}}{[n]_q}\left(1+\frac{2}{[2]_q}\right)\\& = &-x^{2}(1-q(\beta_{n}^{(r)})^{2})+\frac{1}{[3]_q[n]_q^2}+\frac{x\beta_{n}^{(r)}}{[n]_q}\left(1+\frac{2}{[2]_q}\right) .
\end{eqnarray*}
Since $q,\beta_{n}^{(r)} \in (0,1)$, we get
\begin{eqnarray}\label{eq7}
\mathfrak{K}_{n,q}^{\beta^{(1)},\dotsb,\beta^{(r)}}(s^{2};x)-x^{2} & \leq & \frac{1}{[3]_q[n]_q^2}+\frac{x\beta_{n}^{(r)}}{[n]_q}\left(1+\frac{2}{[2]_q}\right) .
\end{eqnarray}
Using the positivity and linearity of the operators, and equation $(\ref{Riman3})$, we have
\begin{eqnarray*}
0 \leq \mathfrak{K}_{n,q}^{\beta^{(1)},\dotsb,\beta^{(r)}}((s-x)^{2};x)&=&\mathfrak{K}_{n,q}^{\beta^{(1)},\beta^{(2)}}(s^{2};x)-2x\mathfrak{K}_{n,q}^{\beta^{(1)},\dotsb,\beta^{(r)}}(s;x)+x^{2}\\
\text{or,}\quad -2x^{2}(1-\beta_{n}^{(r)}) &\leq & \mathfrak{K}_{n,q}^{\beta^{(1)},\dotsb,\beta^{(r)}}(s^{2};x)-x^{2}\\
\text{or,}\quad -2x^{2}(1-\beta_{n}^{(r)})-\frac{1}{[3]_q[n]_q^2}-\frac{x\beta_{n}^{(r)}}{[n]_q}\left(1+\frac{2}{[2]_q}\right) &\leq & \mathfrak{K}_{n,q}^{\beta^{(1)},\dotsb,\beta^{(r)}}(s^{2};x)-x^{2}.
\end{eqnarray*}
Hence, in view of $(\ref{eq7})$
\begin{eqnarray*}
-2x^{2}(1-\beta_{n}^{(r)})-\frac{1}{[3]_q[n]_q^2}-\frac{x\beta_{n}^{(r)}}{[n]_q}\left(1+\frac{2}{[2]_q}\right) &\leq & \mathfrak{K}_{n,q}^{\beta^{(1)},\dotsb,\beta^{(r)}}(s^{2};x)-x^{2}\\& < &\frac{1}{[3]_q[n]_q^2}+\frac{x\beta_{n}^{(r)}}{[n]_q}\left(1+\frac{2}{[2]_q}\right) + 2x^{2}(1-\beta_{n}^{(r)}),
\end{eqnarray*}
thus $$|\mathfrak{K}_{n,q}^{\beta^{(1)},\dotsb,\beta^{(r)}}(s^{2};x)-x^{2}| \leq \frac{1}{[3]_q[n]_q^2}+\frac{x\beta_{n}^{(r)}}{[n]_q}\left(1+\frac{2}{[2]_q}\right) + 2x^{2}(1-\beta_{n}^{(r)}).$$
\end{enumerate}
\end{proof}

\begin{lemma}\label{lem2}
For the operators $\mathfrak{K}_{n,q}^{\beta^{(1)},\dotsb,\beta^{(r)}}(.;x)$, we have the inequality
\begin{eqnarray*}
\bigg|\mathfrak{K}_{n,q}^{\beta^{(1)},\dotsb,\beta^{(r)}}\big((s-x)^2;x\big)\bigg| & \leq & 2x(1+x)(1-\beta_{n}^{(r)}) + \frac{x}{[n]_{q}}\bigg(\beta_{n}^{(r)}\bigg(1+\frac{2}{[2]_q}\bigg)+\frac{2}{[2]_q}\bigg)+\frac{1}{[3]_q[n]_q^2}\\
&\leq & 4(1-\beta_{n}^{(r)}) +\frac{1}{[n]_{q}}\bigg(\beta_{n}^{(r)}\bigg(1+\frac{2}{[2]_q}\bigg)+\frac{2}{[2]_q}\bigg)+\frac{1}{[3]_q[n]_q^2}\\
&=& \gamma_{n,q}(\beta_{n}^{(r)}), \quad \text{say.}
\end{eqnarray*}
\begin{proof}
We can write
\begin{eqnarray*}
\bigg|\mathfrak{K}_{n,q}^{\beta^{(1)},\dotsb,\beta^{(r)}}\big((s-x)^2;x\big)\bigg| &\leq & \bigg|\mathfrak{K}_{n,q}^{\beta^{(1)},\dotsb,\beta^{(r)}}\big(s^2-x^2;x\big)\bigg|+ 2x \bigg|\mathfrak{K}_{n,q}^{\beta^{(1)},\dotsb,\beta^{(r)}}\big(s-x;x\big)\bigg|.
\end{eqnarray*}
Now, using Lemma \ref{lem1}, we obtain the required inequality.
\end{proof}
\end{lemma}

From now onwards in Sections 3 and 4, we assume that $q = \big<q_{n}\big>\in (0,1)$ such that 
$$q_{n} \to 1 \quad \text{and} \quad q_{n}^{n} \to a\in [0, 1) \quad \text{as $n\to \infty$}.$$ 

\section{Deferred weighted A-statistical Approximation process via $\mathfrak{K}_{n,q}^{\beta^{(1)},\dotsb,\beta^{(r)}}$}
Let $M$ be a subset of the set of natural numbers $\mathbb{N}$ and for each $n\in \mathbb{N}$, we define
$$M_n = \{m\in M: m\leq n \}.$$ The density ( or natural density) of the set $M$, denoted by $d(M)$, is defined by the limit(if exists) of the sequence $\big<\frac{|M_{n}|}{n}\big>$. More precisely,
$$d(M) = \lim_{n\to \infty}\frac{|M_{n}|}{n}.$$ There are many ways to define the density of the subsets of  natural numbers and these definitions are playing a pivotal role in the areas of Number Theory and Graph Theory (see \cite{NRai, Pandey}). Any sequence $\big<x_{n}\big> $ is called statistically convergent to $l$ if, for each $\epsilon >0$, we have the following
$$\lim_{n\to \infty}\frac{|\{k\in \mathbb{N}:k\leq n ~\text{and}~|x_{k}-l|\geq \epsilon\}|}{n} = 0.$$ In this case, we write $stat\underset{n \to \infty}{\lim}x_n = l$. The definition shows that every convergent sequence is always statistically convergent, while the converse need not to be true in general. Karakaya et al. \cite{Kara} derived the concept of weighted statistical convergence and the idea later modified by Mursaleen et al. in \cite{Mur}.\newline
Let us assume that $\big<s_{k}\big>$ be a sequence such that $s_k \geq 0$ and
$$S_n = \sum_{k=1}^{n}s_k, \quad s_1 >0,$$
denotes its partial sum. Now, set
$$u_n = \frac{1}{S_n}\sum_{k=1}^{n}s_k x_k, \quad n\in \mathbb{N}. $$ Then, the sequence $\big<x_n\big>$ is called weighted statistically convergent to a number $l$ if, for any given $\epsilon > 0$, the following holds
$$\lim_{n\to \infty}\frac{|\{k\in \mathbb{N}:k\leq S_n ~\text{and}~s_k|x_{k}-l|\geq \epsilon\}|}{S_n} = 0,$$ and we write $stat_w\underset{n\to \infty}{\lim}x_n=l$. If $X_1$ and $X_2$ are sequence spaces such that for every infinite matrix $A=(a_{n,k}):X \to Y$, we have $(Ax)_n = \sum_{k=1}^{\infty}a_{n,k}x_{k}$. Then, the matrix $A$ is called regular if $\underset{n\to \infty}{\lim} (Ax)_n= l$ whenever $\underset{k\to \infty}{\lim} (x)_k= l$. For a non-negative regular matrix $A=(a_{n,k})$, Freedman et al. \cite{Free} defined the idea of $A$-statistical convergence. The sequence $(x)_n$ is called $A$-statistically convergent to a number $l$, denoted by $stat_A\underset{n \to \infty}{\lim}x_n = l$, if
$$\lim_{n \to \infty}\sum_{k:|x_n-l|\geq \epsilon}a_{n,k}=0, \quad ~ \text{for every}~\epsilon >0.$$
Very recently, Srivastava et al. \cite{HM} derived a more general concept of $A$-statistical convergence and called it deferred weighted $A$-statistical convergence. Suppose $(b_n)$ and $(c_n)$ are the sequences of non-negative integers satisfying the regularity conditions $b_n < c_n; \underset{n\to \infty}{\lim}c_n=\infty$. Now, we set
$$S_n = \sum_{m=b_n+1}^{c_n}s_m,$$ for any given sequence $(s_n)$ of non-negative real numbers and its respective deferred weighted mean by $\rho_n= \frac{1}{S_n}\sum_{m=b_n+1}^{c_n}s_mx_m$. Then, the sequence $(x_n)$ is called deferred weighted summable (denoted by $c^{DWS}-\underset{n\to \infty}{\lim}x_n =l$) to $l$ if $\underset{n \to \infty}{\lim}\rho_n = l$. Also, we call $(x_n)$ to be deferred weighted A-summable to (denoted by $c^{DWS}_A-\underset{n\to \infty}{\lim}x_n =l$) a number $l$ if
$$\lim_{n\to \infty}\frac{1}{S_n}\sum_{m=b_n+1}^{c_n}\sum_{k=1}^{\infty}s_ma_{m,k}x_k = l.$$ Let $c^{DWS}$ be the space of all deferred weighted summable sequences and $(b_n),(c_n)$ are the sequences of non-negative integers. Then an infinite matrix $A=(a_{n,k})$, is called deferred weighted regular matrix if
$$(Ax)_n=\sum_{k=1}^{\infty}a_{n,k}x_k\in c^{DWS} \quad \text{for every convergent sequence}~ x=(x_n),$$
with $$c^{DWS}-\underset{n \to \infty}{\lim}(Ax)_n = stat_A \underset{n \to \infty}{\lim}x_n.$$ For a non-negative deferred weighted regular matrix $A=(a_{n,k})$ and $K_{\epsilon} \subset \mathbb{N}=\{k\in\mathbb{N}:|x_k-l|\geq \epsilon\}$, a sequence $\big<x_n\big>$ is said to be deferred weighted $A$-statistically convergent to $l$ (denoted by $stat_{A}^{DW}-\underset{n \to \infty}{\lim}x_n = l$) if, for each $\epsilon >0$, the deferred weighted $A$-density of $K_{\epsilon}$ denoted by $d^A_{DW}(K_{\epsilon})$ is zero. That is
\begin{eqnarray*}
d^A_{DW}(K_{\epsilon})= \lim_{n\to \infty}\frac{1}{S_n}\sum_{m=b_n+1}^{c_n}\sum_{k\in K_{\epsilon}}s_ma_{m,k}= 0.
\end{eqnarray*}
In our further consideration, in this section we assume $A=(a_{n,k})$ to be a non-negative deferred weighted regular matrix. \newline The following theorem shows the deferred weighted $A$-statistical convergence of the operators $\mathfrak{K}_{n,q_n}^{\beta^{(1)},\dotsb,\beta^{(r)}}(.;x)$ defined by $(\ref{eq5})$.
\begin{theorem}
For $f \in \mathfrak{C}(I)$,
\begin{eqnarray*}
stat_{A}^{DW}-\lim_{n\to \infty}\norm{\mathfrak{K}_{n,q_n}^{\beta^{(1)},\dotsb,\beta^{(r)}}(f)-f} &=& 0,
\end{eqnarray*}
if and only if $stat_{A}^{DW}-\underset{n \to \infty}{\lim}\norm{\mathfrak{K}_{n,q_n}^{\beta^{(1)},\dotsb,\beta^{(r)}}(f_i)-f_i}=0$ for $i = 1,2$ where $f_{i}(s)= s^{i}$.
\end{theorem}
\begin{proof}
The necessary part is trivial. For the converse, let us assume that $stat_{A}^{DW}-\underset{n \to \infty}{\lim}\norm{\mathfrak{K}_{n,q_n}^{\beta^{(1)},\dotsb,\beta^{(r)}}(f_i)-f_i}=0$ is true. For $f\in \mathfrak{C}(I)$ there exists a positive constant $M_f$ such that
\begin{eqnarray*}
\big| f(s)-f(x)\big| &\leq & 2M_f, \quad ~ \text{for all}~ s,x\in I.
\end{eqnarray*}
In view of the uniform continuity of $f$ on $I$, for any $\epsilon > 0~ \exists~ \delta >0$, such that $|f(s)-f(x)| < \epsilon$ whenever $|s-x| < \delta$. Hence, for all $s,x \in I$, we can write
\begin{eqnarray}\label{15}
\big| f(s)-f(x)\big| &\leq & \epsilon + \frac{2M_f}{\delta^2}(s-x)^2.
\end{eqnarray}
Thus, applying the operators $\mathfrak{K}_{n,q}^{\beta^{(1)},\dotsb,\beta^{(r)}}(.;x)$ on the above equation, we obtain
\begin{eqnarray*}
\big|\mathfrak{K}_{n,q_n}^{\beta^{(1)},\dotsb,\beta^{(r)}}(f;x)-f(x)\big| & \leq & \epsilon +
\frac{2M_f}{\delta^2}\Big(\big|\mathfrak{K}_{n,q_n}^{\beta^{(1)},\dotsb,\beta^{(r)}}(s^2;x)-x^2\big|+ 2|x|\big|\mathfrak{K}_{n,q_n}^{\beta^{(1)},\dotsb,\beta^{(r)}}(s;x)-x\big|\Big).
\end{eqnarray*}
Using Lemma \ref{lem1} and considering sup norm, we have the following inequality
\begin{eqnarray*}
\norm{\mathfrak{K}_{n,q_n}^{\beta^{(1)},\dotsb,\beta^{(r)}}(f)-f} &\leq & \epsilon + \frac{2M_f}{\delta^2}\Big(\norm{\mathfrak{K}_{n,q_n}^{\beta^{(1)},\dotsb,\beta^{(r)}}(s^2)-x^2}+ 2\norm{\mathfrak{K}_{n,q_n}^{\beta^{(1)},\dotsb,\beta^{(r)}}(s)-x}\Big).
\end{eqnarray*}
Now, for any $\epsilon' > 0$, we consider the following sets;
\begin{eqnarray*}
K_{\epsilon'} &:=& \big{\lbrace}k\in \mathbb{N}: \norm{\mathfrak{K}_{n,q_n}^{\beta^{(1)},\dotsb,\beta^{(r)}}(f)-f} \geq \epsilon'\big{\rbrace};\\
K_{\frac{\epsilon'}{2}} &:=& \bigg{\lbrace}k\in \mathbb{N}:\epsilon +\frac{2M_f}{\delta^2} \norm{\mathfrak{K}_{n,q_n}^{\beta^{(1)},\dotsb,\beta^{(r)}}(s^2)-x^2} \geq  \frac{\epsilon'}{2}\bigg{\rbrace};\\
K^{'}_{\frac{\epsilon'}{2}} &:=& \bigg{\lbrace}k\in \mathbb{N}:\frac{4M_f}{\delta^2}\norm{\mathfrak{K}_{n,q_n}^{\beta^{(1)},\dotsb,\beta^{(r)}}(s)-x} \geq \frac{\epsilon'}{2}\bigg{\rbrace},
\end{eqnarray*}
thus $K_{\epsilon'} \subset K_{\frac{\epsilon'}{2}} \cup K^{'}_{\frac{\epsilon'}{8}}$ and therefore
\begin{eqnarray}\label{13}
\frac{1}{S_n}\sum_{m=b_n+1}^{c_n}\sum_{k\in K_{\epsilon'}}s_ma_{m,k} &\leq & \frac{1}{S_n}\sum_{m=b_n+1}^{c_n}\sum_{k\in K_{\frac{\epsilon'}{2}}}s_ma_{m,k} + \frac{1}{S_n}\sum_{m=b_n+1}^{c_n}\sum_{k\in K^{'}_{\frac{\epsilon'}{2}}}s_ma_{m,k}.
\end{eqnarray}
Now, using the hypothesis and from Lemma \ref{lem1}, it is obvious that
\begin{eqnarray*}
d^A_{DW}(K_{\frac{\epsilon'}{4}}) = \lim_{n \to \infty}\frac{1}{S_n}\sum_{m=b_n+1}^{c_n}\sum_{k\in K_{\frac{\epsilon'}{4}}}s_ma_{m,k} = 0;\\
\text{and}\quad d^A_{DW}(K_{\frac{\epsilon'}{8}}) =\lim_{n \to \infty} \frac{1}{S_n}\sum_{m=b_n+1}^{c_n}\sum_{k\in K^{'}_{\frac{\epsilon'}{8}}}s_ma_{m,k} = 0.
\end{eqnarray*}
Hence, from equation (\ref{13}), we have
\begin{eqnarray*}
d^A_{DW}(K_{\epsilon'}) =\lim_{n\to \infty}\frac{1}{S_n}\sum_{m=b_n+1}^{c_n}\sum_{k\in K_{\epsilon'}}s_ma_{m,k} =  stat_A^D - \lim_{n\to \infty}\norm{\mathfrak{K}_{n,q_n}^{\beta^{(1)},\dotsb,\beta^{(r)}}(f)-f}  = 0.
\end{eqnarray*}
\end{proof}

We recall the definition of modulus of continuity. For any continuous function $f: I \to \mathbb{R}$ and a given $\delta > 0$, the modulus of continuity $\omega_{f}(\delta)$ is defined as
\begin{eqnarray*}
w_{f}(\delta) &:=& \sup_{|s-x|\leq \delta}\{|f(s)-f(x)|: s,x \in I\}.
\end{eqnarray*}
From the above definition, we have
\begin{eqnarray}\label{14}
|f(s)-f(x)| &\leq & \bigg(1+\frac{(s-x)^2}{\delta^2}\bigg)\omega_{f}(\delta).
\end{eqnarray}
Let $\big<\alpha_n\big>$ be a positive non-increasing sequence of real numbers and $\big<\delta_n\big>$ be any sequence of positive real numbers. Then, we say the sequence $\big<\omega_f(\delta_n)\big>$ is deferred weighted A-statistically convergent with $o(\alpha_n)$ if
\begin{eqnarray*}
stat_{A}^{DW} - \lim_{n \to \infty}\frac{\omega_f(\delta_n)}{\alpha_n} = 0.
\end{eqnarray*}

\begin{theorem}
For the operator $\mathfrak{K}_{n,q_n}^{\beta^{(1)},\dotsb,\beta^{(r)}}$, if
\begin{eqnarray*}
\omega_{f}\bigg(\sqrt{\gamma_{n,q_n}(\beta_{n}^{(r)})}\bigg) = stat_{A}^{D}-o(\alpha_{n}),
\end{eqnarray*}
then
\begin{eqnarray*}
\norm{\mathfrak{K}_{n,q_n}^{\beta^{(1)},\dotsb,\beta^{(r)}}(f)-f} = stat_{A}^{D}-o(\alpha_{n}).
\end{eqnarray*}
\end{theorem}
\begin{proof}
From the inequality given in (\ref{14}) and using Lemma \ref{lem1} , for any $\delta > 0$, we obtain
\begin{eqnarray*}
\big|\mathfrak{K}_{n,q_n}^{\beta^{(1)},\dotsb,\beta^{(r)}}(f;x)-f(x)\big| &\leq &  \{1 + \frac{1}{\delta^2}\mathfrak{K}_{n,q_n}^{\beta^{(1)},\dotsb,\beta^{(r)}}((s-x)^2;x)\}\omega_{f}(\delta), \quad \text{for all}~ x\in I.
\end{eqnarray*}
Hence in view of Lemma \ref{lem2}, we can write
\begin{eqnarray*}
\norm{\mathfrak{K}_{n,q_n}^{\beta^{(1)},\dotsb,\beta^{(r)}}(f)-f} &\leq & \omega_{f}(\delta)\bigg(1+\frac{\gamma_{n,q_n}(\beta_{n}^{(r)})}{\delta^2}\bigg).
\end{eqnarray*}
Now, we choose $\delta =\sqrt{\gamma_{n,q_n}(\beta_{n}^{(r)})}$ and consider the hypothesis $\omega_{f}\bigg(\sqrt{\gamma_{n,q_n}(\beta_{n}^{(r)})}\bigg) = stat_{A}^{DW}-o(\alpha_{n})$, to reach the assertion.
\end{proof}

\section{Power Series Summability Approximation Process via $\mathfrak{K}_{n,q_{n}}^{\beta^{(1)},\dotsb,\beta^{(r)}}(.;x)$}
Let $\big<p_{j}\big>$ be a sequence of real numbers such that $p_1 > 0$ and $p_{j}\geq 0$, $\forall~ j = 2,3, \dotsb$.
Also, suppose that the power series
\begin{eqnarray}\label{*}
p(u) &=& \sum_{j=1}^{\infty}p_{j}u^{j-1},
\end{eqnarray}
has a radius of convergence $ R \in (0, \infty]$. Now, a sequence $\big<\eta_{j}\big>$ is said to be convergent to $l$ in the sense of power series method ( please see \cite{Bor, Kra, Sta}) if
\begin{eqnarray*}
\lim_{u \to R-}\frac{1}{p(u)}\sum_{j=1}^{\infty}\eta_{j}p_{j}u^{j-1}&=& l, \quad \forall ~ x\in(0,R).
\end{eqnarray*}
Further, the power series method is called regular \cite{Bos} if and only if
\begin{eqnarray}\label{18}
\lim_{u \to R-}\frac{p_ju^{j-1}}{p(u)} &=& 0, \quad \forall ~ j\in \mathbb{N}.
\end{eqnarray}
Recently, the power series summability method of convergence has attracted the researchers due to its nature of generality over the classical convergence \cite{Tas}. For the interested reader, we refer to \cite{Tas3, Tas2, Tas1}. The following result shows the convergence of our operators $\mathfrak{K}_{n,q_n}^{\beta^{(1)},\dotsb,\beta^{(r)}}(.;x)$ by means of power series method.
\begin{theorem}\label{thm4}
For $f \in \mathfrak{C}(I)$, the operators $\mathfrak{K}_{n,q_n}^{\beta^{(1)},\dotsb,\beta^{(r)}}(.;x)$ satisfy
\begin{eqnarray}\label{16}
\lim_{u \to R-}\frac{1}{p(u)}\sum_{n= 1}^{\infty}\norm{\mathfrak{K}_{n,q_n}^{\beta^{(1)},\dotsb,\beta^{(r)}}(f)-f}p_{n}u^{n-1} = 0,
\end{eqnarray}
if and only if
\begin{eqnarray}\label{17}
\lim_{u \to R-}\frac{1}{p(u)}\sum_{n=1}^{\infty}\norm{\mathfrak{K}_{n,q_n}^{\beta^{(1)},\dotsb,\beta^{(r)}}(f_{i})-f_{i}}p_{n}u^{n-1} = 0,
\end{eqnarray}
for $i = 1,2$ where $f_{i}(s) = s^{i}$.
\end{theorem}

\begin{proof}
First assume that (\ref{16}) is true. Then (\ref{17}) is obvious. Conversely, let the condition (\ref{17}) is true. Now, using the inequality given in (\ref{15}), we can write
\begin{eqnarray*}
\frac{1}{p(u)}\sum_{n= 1}^{\infty}\bigg|\mathfrak{K}_{n,q_n}^{\beta^{(1)},\dotsb,\beta^{(r)}}(f(s);x)-f(x)\bigg|p_{n}u^{n-1} &\leq &  \frac{1}{p(u)}\sum_{n= 1}^{\infty}\bigg|\epsilon + \frac{1}{\delta^2}~\mathfrak{K}_{n,q_n}^{\beta^{(1)},\dotsb,\beta^{(r)}}((s-x)^2;x)\bigg|p_{n}u^{n-1}\\
&\leq & \frac{1}{p(u)}\sum_{n= 1}^{\infty}\bigg{\lbrace}\epsilon + \frac{1}{\delta^2}\bigg(\big|\mathfrak{K}_{n,q_n}^{\beta^{(1)},\dotsb,\beta^{(r)}}(s^2;x)-x^2\big|\\
&+& 2\big|\mathfrak{K}_{n,q_n}^{\beta^{(1)},\dotsb,\beta^{(r)}}(s;x)-x\big|\bigg)\bigg{\rbrace}p_{n}u^{n-1},
\end{eqnarray*}
for all $ x \in I$. Considering sup norm and taking the required limit, in view of $(\ref{*})$, we obtain
\begin{eqnarray*}
\lim_{u \to R-} \frac{1}{p(u)}\sum_{n= 1}^{\infty}\norm{\mathfrak{K}_{n,q_n}^{\beta^{(1)},\dotsb,\beta^{(r)}}(f)-f}p_{n}u^{n-1} &\leq & \epsilon + \frac{1}{\delta^2}\bigg{\lbrace}\lim_{u \to R-} \frac{1}{p(u)}\sum_{n= 1}^{\infty}\norm{\mathfrak{K}_{n,q_n}^{\beta^{(1)},\dotsb,\beta^{(r)}}(s^2)-x^2}p_{n}u^{n-1} \\
&&+\lim_{u \to R-} \frac{1}{p(u)}\sum_{n= 1}^{\infty}\norm{\mathfrak{K}_{n,q_n}^{\beta^{(1)},\dotsb,\beta^{(r)}}(s)-x}p_{n}u^{n-1}\bigg{\rbrace}.
\end{eqnarray*}
Now, the assertion follows easily on using the hypothesis (\ref{17}) and arbitrariness of $\epsilon$.
\end{proof}

\begin{remark}
Let us assume that $\underset{n \to \infty}{\lim}\beta_{n}^{(r)} = 1$. In order to show that uniform convergence of $\big<\mathfrak{K}_{n,q_n}^{\beta^{(1)},\dotsb,\beta^{(r)}}(f)\big>$ to $f$ on $I$ by power series method, it is sufficient to establish the following;
\begin{eqnarray*}
\lim_{u \to R-} \frac{1}{p(u)}\sum_{n= 1}^{\infty}\norm{\mathfrak{K}_{n,q_n}^{\beta^{(1)},\dotsb,\beta^{(r)}}(f_{i})-f_i}p_{n}u^{n-1} &=& 0, \quad \text{for}~ i = 1,2.
\end{eqnarray*}
Using Lemma \ref{lem1},
\begin{eqnarray}\label{19}
\frac{1}{p(u)}\sum_{n = 1}^{\infty}\norm{\mathfrak{K}_{n,q_n}^{\beta^{(1)},\dotsb,\beta^{(r)}}(f_{1};x)-x} p_{n}u^{n-1} &\leq & \frac{1}{p(u)}\sum_{n = 1}^{\infty}\bigg((1-\beta^{(r)}_{n})+\frac{1}{[2]_{q_n}[n]_{q_n}}\bigg) p_{n}u^{n-1}\nonumber\\
&=& J_{1} + J_{2}, \quad \text{say.}
\end{eqnarray}
Now, let us estimate
\begin{eqnarray*}
J_1 &=& \frac{1}{p(u)}\sum_{n = 1}^{\infty}(1-\beta^{(r)}_{n}) p_{n}u^{n-1}.
\end{eqnarray*}
Let $\epsilon > 0$ be an arbitrary. Then from the hypothesis $\exists$ $n_{0}(\epsilon)$ such that $|1-\beta^{(r)}_{n}| \leq \frac{\epsilon}{4}$ for all $n > n_{0}(\epsilon)$. Then
\begin{eqnarray*}
J_1 &\leq & \frac{1}{p(u)}\sum_{n = 1}^{n_0}(1-\beta^{(r)}_{n}) p_{n}u^{n-1} +\frac{\epsilon}{4p(u)}\sum_{n = n_0 +1}^{\infty} p_{n}u^{n-1}\\
& < & \frac{1}{p(u)}\sum_{n = 1}^{n_0}(1-\beta^{(r)}_{n}) p_{n}u^{n-1} +\frac{\epsilon}{4p(u)}\sum_{n =1}^{\infty} p_{n}u^{n-1}.
\end{eqnarray*}
Since $\big<1-\beta_{n}^{(r)}\big>$ is a bounded sequence then $\exists$ $M_1 > 0$ such that $M_1 = \underset{1\leq n\leq n_0}{\max}(1-\beta_{n}^{(r)})$ and using (\ref{*}),
\begin{eqnarray*}
J_ 1 & < & \frac{M_1}{p(u)}\sum_{n = 1}^{n_0} p_{n}u^{n-1} +\frac{\epsilon}{4}.
\end{eqnarray*}
In view of regularity condition given by (\ref{18}) there exists $\delta_{j}(\epsilon) > 0$ such that $\frac{p_{j}u^{j-1}}{p(u)} < \frac{\epsilon}{4M_1n_0}$ for all $R-\delta_{j}(\epsilon) < u < R$, and $j = 1,2,...n_0(\epsilon)$. Let us consider $\delta(\epsilon) = \min\big(\delta_{1}(\epsilon),\delta_{2}(\epsilon), ...,\delta_{n_0}(\epsilon)\big)$ then for every $R-\delta(\epsilon) < u < R$ and for all $n = 1,2,...n_0$, we have
\begin{eqnarray*}
J_{1} & < & \frac{\epsilon}{4M_1n_0} M_1 n_0 + \frac{\epsilon}{4} = \frac{\epsilon}{2}.
\end{eqnarray*}
Now, we estimate
\begin{eqnarray*}
J_2 &=& \frac{1}{p(u)}\sum_{n = 1}^{\infty}\frac{1}{[2]_{q_n}[n]_{q_n}} p_{n}u^{n-1}.
\end{eqnarray*}
Since $\big<\frac{1}{[2]_{q_n}[n]_{q_n}}\big> \to 0$, as $n \to \infty$ there exists $n_{1}(\epsilon)\in \mathbb{N}$ such that $\frac{1}{[2]_{q_n}[n]_{q_n}} < \frac{\epsilon}{4}$, for all $n > n_{1}(\epsilon)$. Let us set $M_2 = \underset{1\leq n \leq n_1(\epsilon)}{\max}\frac{1}{[2]_{q_n}[n]_{q_n}}$. Then
\begin{eqnarray*}
J_2  & <& \frac{1}{p(u)}\sum_{n = 1}^{n_{1}}\frac{1}{[2]_{q_n}[n]_{q_n}} p_{n}u^{n-1} + \frac{\epsilon}{4}\\
& < & \frac{\epsilon}{4M_2n_1} M_2n_1 + \frac{\epsilon}{4} = \frac{\epsilon}{2},
\end{eqnarray*}
for some $\delta^{'}(\epsilon) > 0$ such that $u \in (R-\delta^{'}, R)$, in view of the regularity condition (\ref{18}). Finally on choosing $\delta_{0} = \min\big(\delta, \delta^{'}\big)$ and using the estimates of $I_1; I_2$ in (\ref{19}), we obtain
\begin{eqnarray*}
&&\frac{1}{p(u)}\sum_{n = 1}^{\infty}\norm{\mathfrak{K}_{n,q_n}^{\beta^{(1)},\dotsb,\beta^{(r)}}(f_{1})-f_1} p_{n}u^{n-1} ~ < ~  \epsilon , \quad \forall ~ u\in(R-\delta_0 , R).
\end{eqnarray*}
Next, using Lemma \ref{lem1}, we consider
\begin{eqnarray}\label{20}
\frac{1}{p(u)}\sum_{n = 1}^{\infty}\norm{\mathfrak{K}_{n,q_n}^{\beta^{(1)},\dotsb,\beta^{(r)}}(f_{2};x)-f_2} p_{n}u^{n-1} &\leq &\frac{1}{p(u)}\sum_{n = 1}^{\infty}\bigg(\frac{1}{[3]_{q_n}[n]^2_{q_n}}+\frac{\beta^{(r)}_{n}}{[n]_{q_n}}\bigg(1+\frac{2}{[2]_{q_n}}\bigg)+2(1-\beta^{(r)}_{n})\bigg)p_{n}u^{n-1}\nonumber\\
&=& K_1+K_2+K_3, \quad \text{say}.
\end{eqnarray}
Let the $\epsilon > 0$ be given. Now, first we estimate
\begin{eqnarray*}
K_1 &=&  \frac{1}{p(u)}\sum_{n = 1}^{\infty}\frac{1}{[3]_{q_n}[n]^2_{q_n}}p_{n}u^{n-1}\\
& < & \frac{1}{p(u)}\sum_{n = 1}^{n_2}\frac{1}{[3]_{q_n}[n]^2_{q_n}}p_{n}u^{n-1} + \frac{\epsilon}{6}, \quad \forall ~ n > n_2(\epsilon)\\
&< & \bigg(\max_{1\leq n \leq n_2}\frac{1}{[3]_{q_n}[n]^2_{q_n}}\bigg)\frac{1}{p(u)}\sum_{n = 1}^{n_2}p_{n}u^{n-1} + \frac{\epsilon}{6}\\
&= & \frac{M_3}{p(u)}\sum_{n = 1}^{n_2}p_{n}u^{n-1} + \frac{\epsilon}{6}\\
&< & M_3n_2 \frac{\epsilon}{6M_3n_2} + \frac{\epsilon}{6} = \frac{\epsilon}{3},
\end{eqnarray*}
for all $u\in (R-\theta, R)$ and some $\theta(\epsilon)>0$. Similarly, we can show that there exist some $\theta^{'} > 0$ and $\theta^{''} > 0$ such that
\begin{eqnarray*}
K_2 &= & \frac{1}{p(u)}\sum_{n = 1}^{\infty}\frac{\beta^{(r)}_{n}}{[n]_{q_n}}\bigg(1+\frac{2}{[2]_{q_n}}\bigg)p_{n}u^{n-1} ~<~ \frac{\epsilon}{3}, \quad \forall~ u\in (R-\theta^{'}, R),
\end{eqnarray*}
and
\begin{eqnarray*}
K_3 &=& \frac{2}{p(u)}\sum_{n = 1}^{\infty}(1-\beta^{(r)}_{n}) p_{n}u^{n-1} ~<~ \frac{\epsilon}{3}, \quad \forall~ u\in (R-\theta^{''}, R).
\end{eqnarray*}
Considering $\theta_{0}(\epsilon) = \min\{\theta, \theta^{'}, \theta^{''}\}$, and using the estimates $K_1 - K_3$ in (\ref{20}), we reach
\begin{eqnarray*}
\frac{1}{p(u)}\sum_{n= 1}^{\infty}\norm{\mathfrak{K}_{n,q_n}^{\beta^{(1)},\dotsb,\beta^{(r)}}(f_{2})-f_2}p_{n}u^{n-1}  ~ < ~  \epsilon , \quad \forall ~ u\in(R-\theta_0 , R).
\end{eqnarray*}
\end{remark}

\begin{remark}
Suppose that $\underset{n \to \infty}{\lim}\beta_{n}^{(r)} = 1$. For $f\in \mathfrak{C}(I)$, consider the following sequence of auxiliary operators defined by
\begin{eqnarray}\label{remeq1}
\mathfrak{P}_{n,q_n}^{\beta^{(1)},\dotsb,\beta^{(r)}}(f;x) = (1+x_{n})\mathfrak{K}_{n,q_n}^{\beta^{(1)},\dotsb,\beta^{(r)}}(f;x),
\end{eqnarray}
where $\big<x_n\big>=
\begin{array}{cc}
  \bigg\{\begin{array}{cc}
  1 ,& \text{if}~~ n=m^2, m\in \mathbb{N} ,\\
  0,& \text{otherwise}. \\
    \end{array}
\end{array}
$. Now if we take $p_n= 1
$ for all $n \in \mathbb{N}$, then we obtain $p(u)= \sum_{n=1}^{\infty}p_{n}u^{n-1}= \frac{1}{1-u}, |u| < 1 $ which implies that $R = 1$. Further, we note that 
\begin{eqnarray*}
\frac{1}{p(u)}\sum_{n=1}^{\infty}p_{n}u^{n-1}x_{n} = \frac{(1-u)}{u}\sum_{m=1}^{\infty}u^{m^2}.
\end{eqnarray*}
Since by Cauchy's root test, the series $\sum_{m=1}^{\infty}u^{m^2}$ is absolutely convergent in the interval $|u|<1$, it follows that
\begin{eqnarray}\label{remeq2}
\lim_{u \to 1-}\frac{1}{p(u)}\sum_{n=1}^{\infty}p_{n}u^{n-1}x_{n} = \lim_{u \to 1-} \frac{(1-u)}{u}\sum_{m=1}^{\infty}u^{m^2} = 0.
\end{eqnarray}
Hence, the sequence $\big<x_n\big>$ converges to zero in the sense of power series method. Using the definition of auxiliary operators and (\ref{remeq2}), we conclude that
\begin{eqnarray*}
\lim_{u \to 1-}\frac{1}{p(u)}\sum_{n=1}^{\infty}p_{n}u^{n-1}\norm{\mathfrak{P}_{n,q_n}^{\beta^{(1)},\dotsb,\beta^{(r)}}(f_0)-f_0} = 0.
\end{eqnarray*}   
Moreover, from equation (\ref{remeq1}) and Lemma \ref{lem1}, we have
\begin{eqnarray*}
\norm{\mathfrak{P}_{n,q_n}^{\beta^{(1)},\dotsb,\beta^{(r)}}(f_1)-f_1} &\leq & (1-\beta_{n}^{(r)}) + \frac{1}{[2]_{q_{n}}[n]_{q_n}} + x_{n}\bigg(\beta_{n}^{(r)} + \frac{1}{[2]_{q_{n}}[n]_{q_{n}}}\bigg).
\end{eqnarray*}
Since $\big<(1-\beta_{n}^{(r)}) + \frac{1}{[2]_{q_n}[n]_{q_n}}\big>$ converges to 0 as $n \to \infty $, it will also converge to 0 in the sense of power series method. Further $\beta_{n}^{(r)} + \frac{1}{[2]_{q_n}[n]_{q_n}} \leq  2$ for each $n \in \mathbb{N}$, hence in view of (\ref{remeq2}), we have 
\begin{eqnarray*}
\lim_{u \to 1-}\frac{1}{p(u)}\sum_{n=1}^{\infty}p_{n}u^{n-1}\norm{\mathfrak{P}_{n,q_n}^{\beta^{(1)},\dotsb,\beta^{(r)}}(f_1)-f_1} = 0.
\end{eqnarray*}
Again, using the definition (\ref{remeq1}) and Lemma \ref{lem1}, we obtain
\begin{eqnarray*}
\norm{\mathfrak{P}_{n,q_n}^{\beta^{(1)},\dotsb,\beta^{(r)}}(f_2)-f_2} &\leq & 
\bigg(\frac{1}{[3]_{q_n}[n]^2_{q_n}} +\frac{\beta_{n}^{(r)}}{[n]_{q_n}}\bigg(1+\frac{2}{[2]_{q_n}}\bigg) + 2(1-\beta_{n}^{(r)})\bigg) + x_n\bigg(1+\frac{1}{[n]_{q_n}}\bigg(1+\frac{2}{[2]_{q_n}}\bigg)+\frac{1}{[3]_{q_n}[n]^2_{q_n}}\bigg).
\end{eqnarray*}
We see that the sequence $\big<\frac{1}{[3]_{q_n}[n]^2_{q_n}} +\frac{\beta_{n}^{(r)}}{[n]_{q_n}}\bigg(1+\frac{2}{[2]_{q_n}}\bigg) + 2(1-\beta_{n}^{(r)})\big>$ converges to 0 in the sense of power series method and  $\bigg(1+\frac{1}{[n]_{q_n}}\bigg(1+\frac{2}{[2]_{q_n}}\bigg)+\frac{1}{[3]_{q_n}[n]^2_{q_n}}\bigg) \leq 5$ for each $n \in \mathbb{N}$. Thus using (\ref{remeq2}), we have
\begin{eqnarray*}
\lim_{u \to 1-}\frac{1}{p(u)}\sum_{n=1}^{\infty}p_{n}u^{n-1}\norm{\mathfrak{P}_{n,q_n}^{\beta^{(1)},\dotsb,\beta^{(r)}}(f_2)-f_2} = 0.
\end{eqnarray*}
This confirms that the auxiliary operator defined by (\ref{remeq1}) satisfies all the conditions given by (\ref{17}) of Theorem \ref{thm4}, therefore
\begin{eqnarray*}
\lim_{u \to 1-}\frac{1}{p(u)}\sum_{n=1}^{\infty}p_{n}u^{n-1}\norm{\mathfrak{P}_{n,q_n}^{\beta^{(1)},\dotsb,\beta^{(r)}}(f)-f} = 0.
\end{eqnarray*} 
However, $\big<x_n\big>$ is not convergent to 0 as $n \to \infty$ (in the usual sense). Thus, the Korovkin theorem for linear positive operators does not work for the auxiliary operator defined in (\ref{remeq1}). Hence, our Theorem \ref{thm4} is a non-trivial generalization of the classical Korovkin theorem.             
\end{remark}

\begin{theorem}
For $f\in \mathfrak{C}(I)$, if
\begin{eqnarray*}
\frac{1}{p(u)}\sum_{n= 1}^{\infty}\omega_{f}\big(\sqrt{\gamma_{n,q_n}(\beta_{n}^{(r)})}\big) p_{n}u^{n-1} = O(\Omega(u)), \quad as ~u \to R-
\end{eqnarray*}
then
\begin{eqnarray*}
\frac{1}{p(u)}\sum_{n= 1}^{\infty}\norm{\mathfrak{K}_{n,q_n}^{\beta^{(1)},\dotsb,\beta^{(r)}}(f)-f}p_{n}u^{n-1} = O(\Omega(u)),
\end{eqnarray*}
as $u \to R-$, where $\Omega(u)$ is some positive function on $(0, R)$.
\end{theorem}
\begin{proof}
For $f \in \mathfrak{C}(I)$ and $\delta > 0$, using Lemma \ref{lem2}
\begin{eqnarray*}
\frac{1}{p(u)}\sum_{n = 1}^{\infty}\norm{\mathfrak{K}_{n,q_n}^{\beta^{(1)},\dotsb,\beta^{(r)}}(f)-f}p_{n}u^{n-1} &\leq & \frac{1}{p(u)}\sum_{n = 1}^{\infty}\big{\lbrace}1+\frac{1}{\delta^2}\norm{\mathfrak{K}_{n,q_n}^{\beta^{(1)},\dotsb,\beta^{(r)}}(s-x)^2}\big{\rbrace}\omega_{f}(\delta)p_{n}u^{n-1}\\
&\leq & \frac{1}{p(u)}\sum_{n = 1}^{\infty}\big{\lbrace}1+\frac{1}{\delta^2}\gamma_{n,q_n}(\beta_{n}^{(r)})\big{\rbrace}\omega_{f}(\delta)p_{n}u^{n-1}
\end{eqnarray*}
for every $u \in (0, R)$. Taking $\delta = \sqrt{\gamma_{n,q_n}(\beta_{n}^{(r)})}$, we obtain
\begin{eqnarray*}
\frac{1}{p(u)}\sum_{n = 1}^{\infty}\norm{\mathfrak{K}_{n,q_n}^{\beta^{(1)},\dotsb,\beta^{(r)}}(f)-f}p_{n}u^{n-1} &\leq & \frac{2}{p(u)}\sum_{n = 1}^{\infty}\omega_{f}\big(\sqrt{\gamma_{n,q_n}(\beta_{n}^{(r)})}\big) p_{n}u^{n-1}.
\end{eqnarray*}
Hence in view of our hypothesis, the required result follows.
\end{proof}

\begin{center}
\textbf{Acknowledgments}
\end{center}
The authors are grateful to Dr. E. Erku\c{s}-Duman and Dr. O. Duman for their constructive and invaluable suggestions. We believe these suggestions have enrich the presentation and quality of the paper. The second author  also gratefully acknowledges the financial support given to him by the Ministry of Education, Govt. of India to carry out the above work.
\begin{center}
\textbf{Declarations}
\end{center}
\textbf{Conflict of interest:-} Not applicable;\\
\textbf{Data availability:-} We assert that no data sets were generated or analyzed during the preparation of the manuscript;\\
\textbf{Code availability:-} Not applicable;\\
\textbf{Authors' contributions:-} All the authors have equally contributed to the conceptualization, framing and writing of the manuscript.

\end{document}